\documentclass[reqno,10pt]{amsart}
\usepackage{amscd,amssymb,amsmath,color,url}
\usepackage{latexsym}
\usepackage[all]{xy}
\usepackage{defs}
\usepackage[colorlinks=true]{hyperref}

\usepackage{tikz-cd}

\def\LogBl{{\rm LogBl}}
\def\.{{,\dots,}}
\def\dashto{\dashrightarrow}

\begin{document}
\title{Semistable reduction over thick log points}
\author{Alexander E. Motzkin and Michael Temkin}

\thanks{This research is supported by ERC Consolidator Grant 770922 - BirNonArchGeom and BSF grant 2018193}

\address{Einstein Institute of Mathematics\\
               The Hebrew University of Jerusalem\\
                Edmond J. Safra Campus, Giv'at Ram, Jerusalem, 91904, Israel}\email{alex.motzkin@mail.huji.ac.il}
\email{michael.temkin@mail.huji.ac.il}

\begin{abstract}
We establish a version of a semistable reduction theorem over a log point with a non-trivial nilpotent structure. In order to do this we extend the classical desingularization theories to non-reduced schemes with generically principal nilradical.
\end{abstract}

\maketitle

\section{Introduction}

The main goal of this work is to establish a version of semistable reduction (or relative desingularization) of schemes over a base of the form $B_n=\Spec(k[\pi]/(\pi^n))$, where $k$ is a field of characteristic zero. In particular, we were asked about existence of such a result by Christian Schnell and Junchao Shentu. The original hope and motivation was that a strong enough solution of this problem could lead to a progress towards invariance of plurigenera for compact Kaehler manifold. Unfortunately, our results do not provide this and, moreover, indicate (though not prove) that the original hopes were too optimistic, see Remark~\ref{optimistic}. 

We explain in \S\ref{logsmsec} which type of semistable reduction over $B_n$ we construct in the end and argue why this is the natural one and, probably, the strongest one which exists in general. Obstacles for stronger resolutions are not studied in the paper, but it seems plausible that they are related to non-triviality of (a correct version of) birational geometry of log schemes over a log point (e.g. snc schemes).

Note that a general theory of resolution of relative (log) morphisms $X\to B$ was recently developed in \cite{ATW} under the assumption that $B$ is log regular, so it cannot be directly applied to our case. There is a way to generalize that theory to the base we need, but it will be worked out elsewhere, and in this work we decided to use more standard tools. In fact, we manage to deduce our main result from the classical resolution and principalization theorems.

\subsection{Log smoothness over thick points}\label{logsmsec}
We would like to stress that a priori it is not even so clear how to formulate an appropriate semistable reduction conjecture whose goal is to resolve a generically smooth morphism $X\to B=B_n$. For comparison, recall the situation with semistable reduction over a trait, say, over $S=\Spec(R)$ with $R=k[\pi]_{(\pi)}$. If $Z\to S$ has a connected but not irreducible closed fiber $Z_s$, for example, $Z=\Spec(R[x,y]/(xy-\pi))$, then by Zariski's connectedness theorem the same will be true for any modification $Z'$ of $Z$, and hence $Z'\to S$ is not smooth. The best one can do is to resolve $Z$ and make $Z_s$ an snc divisor. In the zero characteristic case this improves the morphism $Z\to S$ drastically: \'etale-locally $Z$ becomes isomorphic to the model charts of the form $\Spec(R[t_1\.t_n]/(t_1^{d_1}\dots t_n^{d_n}-\pi))$, in particular, it is log smooth with respect to the natural closed fiber log structures. For the sake of completeness, we note that such a morphism is called semistable if $d_i\in\{0,1\}$, and one can further make $Z\to S$ semistable by an extension of $R$, normalizing the pullback of $Z$ and then logarithmically (or toroidally) blowing up this nolrmalized pullback, but we will not need this.

Adopting the logarithmic approach is also absolutely critical for the results of \cite{ATW}, so it is natural to adopt it in our case too and provide $B$ with the log structure induced by $\bbN\log\pi$. Nevertheless, it turns out that one cannot in general modify a generically smooth $B$-scheme $X$ to $X'$ which is log smooth over $B$. Here is the simplest example: $X=\Spec(k[x,y]/(y^n))$ with the morphism $X\to B$ given by $\pi=xy$. Indeed, the morphism is not smooth at $x=0$ because the nilradical of $\calO_X$ is not generated by $\pi$ at this point. The situation does not improve if we increase the log structure on $X$ by adding $\bbN\log x$. Moreover, any modification $X'\to X$ induces an isomorphism of the reductions and only increases the nilpotent structure (for example, by replacing $y$ by $y'=y/x^n$), hence it does not improve the situation either. The only way to achieve some sort of resolution of $X\to B$ is to use the whole class of smooth $B$-curves, including the reducible ones, and there are plenty of them. For example, $Y=\Spec(k[x,y]/(x^ny^n))$ is log smooth over $B$ (where $\pi=xy$), and $X$ is an irreducible component of $Y$. In fact, this example is the closed fiber of the morphism $Z\to S$ from the example in the above paragraph, but reducible closed fibers over $S$ show up naturally, while reducible schemes over $B$ at first glance look as a pathology, and by no means they can show up as modifications of irreducible ones.

We hope that the above discussion and relation to the classical semistable reduction over $S$ give enough motivation for our definition: a desingularization of a $B$-scheme $X$ is a modification $X'\to X$ such that locally $X'$ can be realized as a monomial subscheme of codimension zero of a log smooth $B$-scheme $Y$. In other words, $X'$ is locally defined in $Y$ by vanishing of monomials and its reduction is a reduced irreducible component of $Y$. Moreover, for our needs it will suffice to work with a log smooth $Y$ of a special combinatorial type, see \S\ref{intrsec} below.

\begin{rem}
Here is another argument in favor of our definition of desingularization over $B$. We have already mentioned that principalization of ideals on log smooth $B$-schemes is possible (and will be developed elsewhere). Note only that principalization over a log regular $B$ in \cite{ATW} involves modifications of the base, but if $B$ is a trait with the maximal log structure, then no non-trivial modifications exist and we work over the original $B$. Similarly, if $B$ is arbitrary, then principalization over it is possible only once $B$ is replaced by a modification of a log blowing up. For example, the base $B=\Spec(k)$ with the hollow log structure $\bbN\log x+\bbN\log y$ can be replaced by the log blowing up along $(\log x,\log y)$, which is isomorphic to $\bbP^1_k$ (see \cite[Example~4.23(i)]{lognotes}). However, if $B=B_n$, then it possesses no non-trivial modifications or log blowings up, so the principalization takes place over it.

As usually, the principalization over $B$ can be used to obtain a desingularization of $B$-schemes $X$ as follows. Separate components, reducing to the case when $X$ is irreducible. Locally embed $X$ into a log smooth $B$-scheme $M$. Principalize $\calI_X$ on $M$ by blowing up centers which are monomial on $B$-log smooth subschemes of $M$ and stop when the new blowing up center $V$ contains the generic point of the strict transform $X'$ of $X$. In particular, $V$ is log smooth over $B$ and $X'$ is its (non-necessarily reduced) irreducible component. In the situation of \cite{ATW} (similarly to the classical resolution) the connected components of $V$ are irreducible, hence $X'$ is one of them and $X'\to X$ is a $B$-resolution. In our situation, a log smooth $V$ does not have to be irreducible, hence we only obtain a modification $X'$ of $X$ which is an irreducible component of a log smooth $V$. This is precisely the above definition of a $B$-desingularization.
\end{rem}

\begin{rem}\label{optimistic}
(i) One can wonder if there exists another type of resolutions: a proper morphism $X'\to X$ which is birational on one component of $X'$ and such that $X'\to B$ is log smooth. We do not study this question in the paper, but it seems plausible that in general this is impossible.

(ii) One can also ask if the embedding of $X'$ as a (non-reduced) irreducible component of a $B$-log smooth $Y$ can be made globally. This also seems to be a not so natural question, though we show in the paper that this is possible, and if $X\to B$ is proper, then $Y$ can also be taken to be proper over $B$. This construction uses log blowings up, which are natural non-birational proper morphisms $X'\to X$ (which contract components). However, log blowings up cannot be used to attack the question in (i) because they are log \'etale, and hence if $X\to B$ is not log smooth, then the same is true for $X'\to B$.
\end{rem}

\subsection{Main results}\label{intrsec}
As in the main body of the paper we will use a down to earth description which does not involve log geometry. A pair $(X,E)$, where $X$ is a $B$-scheme of finite type and $E$ is an effective Cartier divisor on $X$, is called {\em distinguished} (see \S\ref{Bpairs}) if each point of $X$ possesses a neighborhood $X'$ with $E'=E\times_XX'$ and an \'etale morphism 
$$h\:X'\to X_0=\Spec(k[\veps,t_1\.t_m]/(\veps^n))$$ given by $\pi=\veps t_1^{d_1}\dots t_r^{d_r}$ and such that $E'=V(t_1\dots t_r)$. In particular, $X_0$ is a monomial irreducible component of the log smooth $B$-scheme $Y_0=\Spec(C)$ given by $\veps^n=0$, where $$C=k[\pi,\veps,t_1\.t_m]/(\pi^n,\pi-\veps t_1^{d_1}\dots t_r^{d_r})=k[\veps,t_1\.t_m]/(\veps^nt_1^{nd_1}\dots t_r^{nd_r}).$$ 

\begin{rem}
An \'etale morphism to a closed subscheme locally extends to a scheme, hence shrinking $X'$ further if necessary we can factor $X'\to Y_0$ into composition of a closed immersion $X'\into Y'$ and an \'etale morphism $Y'\to Y$. Thus, locally $X$ is isomorphic to a monomial irreducible component $X'$ of a log smooth $B$-scheme $Y'$ with the log structure of a very special form: it is generated over the log structure $\bbN\log\pi$ of $B$ by $\log(\veps),\log(t_1)\.\log(t_m)$ subject to the relation $\log(\veps)+\sum_{i=1}^rd_i\log(t_i)=\log(\pi)$.
\end{rem}

Now let us formulate a light version of our main result -- Theorem~\ref{mainth1}.

\begin{theor}\label{mainth01}
Let $k$ be a field of characteristic zero and $B=\Spec(k[\pi]/(\pi^n))$. Then there exists a procedure $\calF$ compatible with smooth morphisms $Y\to X$ which obtains as an input a generically smooth morphism of finite type $X\to B$ and outputs a modification $\calF(X)\:X'\to X$ with a Cartier divisor $E'$ such that the $B$-pair $(X',E')$ is distinguished.
\end{theor}

In addition, we will show in Theorem \ref{mainth1} that one can choose $\calF$ compatible with an initial boundary $E$ on $X$ and such that it makes monomial a given closed subscheme $Z\into X$. So, this also includes an analogue in our setting of the classical principalization theorem.

\begin{rem}
One can also wonder if it is possible to achieve that $d_i\in\{0,1\}$. As in the classical case, this is possible after base change with respect to a morphism $\Spec(k[\pi']/(\pi'^{dn}))=B'\to B$ sending $\pi$ to $\pi'^d$, where $d$ is large enough. The proof is purely combinatorial -- one constructs a log blowing up $X''\to X'=X\times_BB'$ through a polytope subdivision provided by the main combinatorial result of \cite{KKMS}. A down to earth way is to copy the argument in the proof of \cite[Theorem~4.4.4]{Temkin-embedded}. An alternative is just to apply \cite[Theorem~4.4]{semistable} to the morphism $X\to B$, as a monoidal alteration of $B$ in the sense of \cite{semistable} is precisely a morphism $B'\to B$ as above. Note also that in our case $\oM_B=\bbN$, so the combinatorial ingredient is again the main result of \cite{KKMS} (though for a general $B$ \cite[Theorem~4.4]{semistable} is based on its generalization to maps of polytopes). 
\end{rem}

\subsection{Outline of the paper}
Now, let us briefly describe our other results and the structure of the paper. In Section~\ref{thicksec} we extend some resolution results to non-reduced schemes. For the sake of completeness we recall in \S\ref{Hirsec} Hironaka's resolution of non-reduced schemes by {\em thick manifolds} -- schemes $X$ with a regular reduction $\tilX=X^\red$ and normally flat along $\tilX$. However, we only need to work with schemes $X$ whose nilradical is generically principal, and their resolution is deduced in \S\ref{desingsec} from the classical functorial principalization: in Theorem~\ref{nonembeddedth} we show that any such $X$ possesses a functorial modification $X'$, which is a principally thick manifold or a ptm -- a thick manifold with a locally principal nilradical, see \S\ref{ptmsec}.

In \S\ref{princsec} we deduce principalization of subschemes of a ptm $X$ from the classical principalization of subschemes of $\tilX$, and in \S\ref{Bschemes} we prove our main result on resolution of varieties over $B$. As in the case of $S$-schemes, we just resolve $X$ (by a thick manifold since $X$ is non-reduced) and then principalize the divisor $(\pi)$ on $X$. There are small additional subtleties in this case, which are dealt with in \S\ref{monomsec}. Finally, in \S\ref{othersec} we explain how analogous results can be proved for certain quasi-excellent schemes, formal schemes and analytic spaces.

Section \ref{mainsec} is devoted to proving that in the case of varieties, after an additional work the resolved $B$-scheme $X'$ can be globally embedded into a log smooth $B$-scheme $Y$. This operation is less functorial, it depends on the choice of a retract, and our argument does not apply to general qe schemes or analytic varieties, because a retract does not have to exist in these cases. It might be the case that the generalization of our main result we prove in Theorem~\ref{mainth} addresses a not so natural question. Still, the intermediate results we prove about retracts and factorization of birational morphisms of non-reduced schemes can be of an independent interest.

\subsection{Conventions}
All log schemes are always assumed to be fs. By a {\em manifold} we mean a regular noetherian scheme or a smooth variety, if it is clear from the context that we are working with varieties. Usually this is used in the context of resolution of singularities, when a singular scheme is embedded into an ambient manifold.

\section{Desingularization on principally thick manifolds}\label{thicksec}
In this section we will extend the classical resolution theory to schemes which do not have to be generically reduced but whose nilradical is generically principal. Everything will be deduced rather straightforwardly from the classical theory.

\subsection{Hironaka's resolution of non-reduced schemes}\label{Hirsec}
In the famous 1964 paper \cite{Hironaka} Hironaka also established desingularization of arbitrary non-reduced varieties that we are going to recall briefly.

\subsubsection{Thick manifolds}
Let us fix some terminology first. We warn the reader that the notion of thick manifolds is non-standard.

\begin{defin}\label{thickdef}
(i) A scheme $X$ is {\em normally flat} along a subscheme $Y$ given by an ideal $\calI\subseteq\calO_X$ if $\calI^n/\calI^{n+1}$ is a flat module over $\calO_Y=\calO_X/\calI$ for any $n$.

(ii) A locally noetherian scheme $X$ is called a {\em thick manifold} if its reduction $X^\red$ is regular and $X$ is normally flat along $X^\red$.

(iii) By {\em thickness} of a noetherian scheme $X$ at a point $x$ we mean the sum $h_X(x)$ of dimensions of the $k(x)$-vector spaces $(\calN^n/\calN^{n+1})\otimes k(x)$, where $\calN$ is the nilradical of $\calO_X$.
\end{defin}

The normal-flatness condition means that although there might be a non-reduced structure, it behaves ``as nice as possible''.

\begin{lem}\label{thicklem}
(i) If $\eta\in X$ is a generic point, then $h_X(\eta)$ is the length of the Artin ring $\calO_{X\eta}$.

(ii) The thickness function $h_X\:X\to\bbN$ is upper-semicontinuous, and $X$ is normally flat along $X^\red$ if and only if $h_X$ is a locally constant function. In the latter case, $X$ has no embedded components.
\end{lem}
\begin{proof}
The first claim is obvious. In (ii) set $Y=X^\red$ and note that for a finite $\calO_Y$-module $M$ the function $r_M(x)=\dim_{k(x)} M\otimes k(x)$ is upper semicontinuous on $Y$, and $r_M$ is locally constant if and only if $M$ is flat. The equivalence in (ii) follows by applying this to $h_X=\sum_nr_{\calN^n/\calN^{n+1}}$ -- due to the semicontinuity, $h_X$ is locally constant if and only if each summand is. Finally, if $X$ has an embedded component, then there exists a closed immersion $X'\into X$ such that the locus $V$ where $X'\neq X$ is nowhere dense but non-empty. In this case, $h_X=h_{X'}$ outside of $V$ and $h_{X}$ jumps on $V$. In particular, for any generic point $\eta$ with a specialization $x\in V$ one has $h_X(\eta)=h_{X'}(\eta)\le h_{X'}(x)<h_X(x)$ and hence $X$ is not a thick manifold.
\end{proof}

\subsubsection{Hironaka's theorem}\label{Hironaka}
In a sense, thick manifolds are non-reduced schemes which are as close to regular ones as possible, and desingularization of non-reduced schemes by thick manifolds is the best one can hope for -- one has to keep the nilpotent structure at the generic points, but it becomes as simple as possible everywhere else.

\begin{defin}\label{moddef}
(i) Throughout this paper a morphism $X'\to X$ is called a {\em modification} if it is proper and induces an isomorphism $U'=U$ of dense open subschemes.

(ii) A {\em desingularization} of a scheme $X$ is a modification $X'\to X$ such that $X'$ is a thick manifold.
\end{defin}

Hironaka proved that this desingularization is, indeed, possible.

\begin{theor}[{\cite[Main Theorem $\rm I^*$]{Hironaka}}]
There exists a construction which associates to any variety $X$ of characteristic zero a projective desingularization $\calF(X)=X'\to X$
\end{theor}

In fact, Hironaka proved his theorem for the more general class of schemes of finite type over a local quasi-excellent ring of characteristic zero.

\subsubsection{Strong resolution}
Hironaka's proof is existential, but modern methods that grew from his approach can also provide a smooth-functorial desingularization. Unfortunately, a convenient reference is missing in the literature, so let us describe the situation briefly. Nowadays there are two main resolution methods, both originating from Hironaka's work, that will be called {\em basic} (or weak in the literature) and {\em strong}. Basic resolution is achieved by blowing up smooth centers in the ambient manifold $M$, but the induced blowings up of the scheme $X$ itself may have singular centers. Its primary invariant is the order of an ideal $I_X\subset\calO_M$. This resolution was studied in the majority of recent works and it is smooth-functorial by results of W{\l}odarczyk, see \cite{Wlodarczyk}.

Strong resolution uses the whole Hilbert-Samuel function as the primary invariant and it only blows up smooth centers such that $X$ is normally flat along them. The price one has to pay is that the process is much heavier, as well as the proofs. In fact, one reduces this resolution to the basic one by encoding the HS function in terms of the order of another ideal (called a presentation). So far, the only known way to resolve general non-reduced schemes is via a strong resolution methods. A canonical strong desingularization was established by Bierstone-Milman, see \cite[Theorem~11.14]{Bierstone-Milman} (though it was only formulated in the embedded version), and its functoriality properties were briefly checked in \cite[Theorem 6.1. Addendum]{BMT} (one should check that the reduction to a basic method is smooth-functorial).

However, we will see that already the basic method suffices to resolve varieties with generically principal nilradical, and this is the only case we will need later. In particular, smooth-functoriality is then automatic.

\subsection{Principally thick manifolds}\label{ptmsec}
We say that a scheme $X$ is {\em nilprincipal} if its nilradical is a locally principal ideal. If, in addition, $X$ is a thick manifold, then we say that $X$ is a {\em principally thick manifold} or a {\em ptm}.

% we mean a scheme $X$nilprincipal thick manifold.scheme $X$ such that $X^\red$ is regular, the nilradical of $X$ is generically principal, and $X$ has no embedded components. We will show in Lemma~\ref{nilprincipallem} that such an $X$ is nothing else but a thick manifold with a locally principal nilradical, but what we will really need in the sequel is the following local description of such schemes.

\begin{defin}
Let $X$ be a ptm. By a {\em family of parameters} at a point $x\in X$ we mean a family $\veps,t_1\.t_d\in\calO_{X,x}$ such that $\veps$ generates the nilradical of $\calO_{X,x}$ and $t_1\.t_d$ induce a family of regular parameters of $\calO_{X^\red,x}=\calO_{X,x}/(\veps)$. We call $\veps$ a {\em nilpotent parameter} and call $t_1\.t_d$ {\em regular parameters}.
\end{defin}

Thickness on a ptm has an especially simple interpretation.

\begin{lem}\label{thickness}
Let $X$ be a ptm, $x\in X$ a point and $\veps\in\calO_x$ a nilpotent parameter at $x$. Then the thickness $h_X(x)$ is the minimal number $h$ such that $\veps^h=0$.
\end{lem}
\begin{proof}
Let $\calN$ be the nilradical, then $\calN_x=\veps\calO_x$ and $\calN^n/\calN^{n+1}$ is non-zero if and only if $\calN^n\neq 0$, that is $n<h$. In the latter case, $(\calN^n/\calN^{n+1})\otimes k(x)$ is generated by the image of $\veps^n$ and hence is of dimension 1.
\end{proof}

\begin{rem}
(i) Due to our definition, ptms of thickness 1 are just usual manifolds (or regular schemes) with dummy nilpotent parameter $\veps=0$. Because of this they should often be dealt with slightly differently, and since all results for them are standard we will simply exclude this case when this is convenient.

(ii) We will later enhance ptms to log schemes with the log structure generated by $\log(\veps)$. In this situation, the case of thickness 1 is not exceptional anymore, in particular, the corresponding log scheme is not log regular.
\end{rem}

\begin{lem}\label{cotangentlem}
If $X$ is a non-reduced ptm and $\veps,t_1\.t_d$ is a family of parameters at $x$, then their images form a basis of the cotangent space $m_x/m_x^2$.
\end{lem}
\begin{proof}
Since $X$ is singular and of dimension $d$ at $x$, we have $\dim(m_x/m_x^2)>d$. The $d+1$ parameters generate $m_x$, hence their images generate $m_x/m_x^2$ and by a dimension consideration they form a basis.
\end{proof}

\begin{lem}\label{thickmanlem}
Let $M$ be a regular scheme and $X\into M$ a closed subscheme given by an ideal $\calI\subset\calO_M$, $x\in X$, $\dim_x(M)=d$ and $\dim_x(X)=l$. Then $X$ is a non-reduced ptm of thickness $n$ locally at a point $x\in X$ if and only if there exists a family of local parameters $x_1\.x_d\in\calO_{M,x}$ such that $\calI_x=(x_1^n,x_2\.x_r)$ for $r=d-l$. In addition, if these conditions are satisfied and $(\veps,t_1\.t_l)$ is a family of parameters of $\calO_{X,x}$, then one can choose $x_1,x_{r+1}\.x_d$ to be any lift of $(\veps,t_1\.t_l)$.
\end{lem}
\begin{proof}
The inverse implication is obvious since $X^\red$ is locally given by the vanishing of $(x_1\.x_r)$, and the nilradical of $\calO_{X,x}$ is generated by the image of $x_1$.

Conversely, assume that $X$ is a ptm and fix any family $(x_1,x_{r+1}\.x_d)$ of elements of $\calO_{M,x}$ that reduces to a family of parameters $(\veps,t_1\.t_l)$ at $x$. The images of $(x_1,x_{r+1}\.x_d)$ in $m_{M,x}/m_{M,x}^2$ are linearly independent because their images in $m_{X,x}/m_{X,x}^2$ form a basis by Lemma~\ref{cotangentlem}. The kernel of the map $m_{M,x}/m_{M,x}^2\onto m_{X,x}/m_{X,x}^2$ is the image of $\calI_x$, hence the family $(x_1,x_{r+1}\.x_d)$ can be completed to a family of regular parameters $(x_1\.x_d)$ so that $x_2\.x_r\in\calI_x$.

Let $n$ be the thickness of $X$ at $x$. Then locally at $x$ the ideal $\calI'=(x_1^n,x_2\.x_r)$ is contained in $\calI$, hence defines a ptm $X'$ containing $X$. We claim that the surjection $\phi\:\calO_{X',x}\onto\calO_{X,x}$ is an isomorphism. Indeed, any element $a\in\Ker(\phi)$ is nilpotent because $X^\red$ and $X'^\red$ are smooth of the same dimension $l=d-r$ at $x$, and using that $X$ and $X'$ are of thickness $n$ we obtain that $a=0$.
\end{proof}

We say that a ptm (or a scheme) $X$ is {\em locally embeddable (in a regular scheme)} if any point $x\in X$ has a neighborhood isomorphic to a closed subscheme of a regular scheme $M$. For example, any $k$-variety satisfies this property. The above lemma implies that locally embeddable ptms are isomorphic to a divisor $nD$ with a regular $D$ in an ambient regular scheme:

\begin{cor}\label{thickmancor}
If $X$ is a locally embeddable ptm, then any point $x\in X$ possesses a neighborhood $U$ and a closed immersion $U\into M$, such that $M$ is regular and $U$ is the divisor in $M$ of the form $V(t^n)$, where $t$ is a parameter of $\calO_{M,x}$ and $n$ is the thickness of $X$ at $x$.
\end{cor}
\begin{proof}
By Lemma~\ref{thickmanlem}, $x$ has a neighborhood $U$ isomorphic to a closed subscheme $V(t_1^n,t_2\.t_r)$ of a regular scheme $M'$, where $t_1\.t_d$ is a regular family of parameters of $\calO_{M',x}$. It remains to take $M=V(t_2\.t_r)$ and define $t\in\calO_{M,x}$ to be the image of $t_1$.
\end{proof}

The following corollary certainly holds for arbitrary ptms, but we only need the case of locally embeddable ones and then the argument is very simple.

\begin{cor}\label{CMcor}
Any ptm $X$, which is locally embeddable in a regular scheme, is Cohen-Macaulay. In particular, if $i\:U\into X$ is an open immersion, whose complement is of codimension at least 2, then $\calO_X=i_*\calO_U$.
\end{cor}
\begin{proof}
By Corollary \ref{thickmancor}, locally $X$ is a Cartier divisor in a smooth (hence CM) variety. Therefore, $X$ is itself CM. The second claim is, perhaps, better known for normal schemes, but it holds more generally for $S_2$-schemes and, in fact, provides an equivalent characterization of this property. The original source is the theory of $Z$-purity, see \cite[Th\'eor\`eme~IV.5.10.5]{ega}, a modern treatment (in the relative setting) is in \cite[Proposition~3.5]{Hassett-Kovacs}.
\end{proof}

\subsubsection{Other characterizations}
For the sake of completeness we provide alternative characterizations of ptms, though this will not be used in the sequel.

\begin{lem}\label{nilprincipallem}
Given a scheme $X$ the following conditions are equivalent:

(i) $X$ is a ptm.

(ii) $X$ is a generically nilprincipal thick manifold.

(iii) $X^\red$ is regular, $X$ is nilprincipal and without embedded components.
\end{lem}
\begin{proof}
(i) implies (ii) and (iii) because thick manifolds have no embedded components by Lemma~\ref{thicklem}(ii). In the sequel, let $\calN$ denote the nilradical of $\calO_X$.

(ii)$\Longrightarrow$(i) Looking at the generic points we see that the locally free module $\calN/\calN^2$ is of rank 1. Its local generator at a point $x$ lifts to a local generator of the stalk $\calN_x$ and hence $X$ is nilprincipal.

(iii)$\Longrightarrow$(i) We will work locally at $x\in X$. Let $\veps\in\calO_x$ be a generator of $\calN_x$ and let $n$ be the minimal number with $\veps^n=0$. We should prove that $X$ is normally flat along $X^\red$ at $x$ and we claim that, in fact, each $\calN_x^i/\calN_x^{i+1}$ with $i<n$ is free of rank 1 over $\calO_{X^\red,x}=\calO_{X,x}/\calN_x$. Clearly, $\calN_x^i/\calN_x^{i+1}$ is generated by the image of $\veps^i$, so we should prove that its annihilator in $\calO_{X,x}/\calN_x$ is zero. Moreover, it suffices to prove the latter only for $i=n-1$, and since $\calN_x^n=0$ this happens if and only if the annihilator of $\veps^{n-1}$ in $\calO_{X,x}$ is contained in (in fact, equal to) $\calN_x$. The latter holds because otherwise $X$ would contain an embedded component on whose complement $\veps^{n-1}$ vanishes.
\end{proof}

\begin{rem}
Note that $X$ is generically nilprincipal if and only if for any generic point $\eta\in X$ the Artin ring $\calO_\eta$ has principal maximal ideal $m_\eta=(\veps)$. Such Artin rings are very special, e.g. see \cite[Proposition~8.8]{AM}, in addition, they can be described as quotients of DVR's by non-zero ideals.
\end{rem}

\subsection{Desingularization of generically nilprincipal varieties}\label{desingsec}
Now we are ready to formulate the resolution result for non-reduced schemes with simplest nilpotent structure.

\begin{theor}\label{nonembeddedth}
There exists a construction which associates to any generically nilprincipal variety $X$ of characteristic zero a projective desingularization $\calF(X)=X'\to X$ which depends on $X$ smooth-functorially: for any smooth morphism $Y\to X$ one has that $\calF(Y)=\calF(X)\times_XY$.
\end{theor}
\begin{proof}
Step 1. {\em Assume that $X$ is irreducible and embedded as a closed subscheme in a smooth variety $M$.} In this case we apply the smooth-functorial principalization from \cite{Wlodarczyk} to $\calI_X\subset\calO_M$ (one can use also the methods of \cite{Kollar} or \cite{Bierstone-Milman-funct}, but they all are equivalent). It outputs a sequence of blowings up $f_i\:M_i\to M_{i-1}$, $1\le i\le n$ with smooth centers $V_i\subset M_i$ and snc boundaries $E_i\subset M_i$ such that $E_0=\emptyset$, $M_0=M$, $E_i=f_i^{-1}(E_{i-1})\cup f_i^{-1}(V_{i-1})$, $V_i$ has simple normal crossings with $E_i$ and lies in the preimage of $X$, and $\calI_X\calO_{M_n}$ is invertible and supported on $E_n$. As one always does when deducing resolution from principalization, we cut this sequence at $M_l$ such that $f_{l+1}$ is the first blowing up whose center $V_l$ maps onto $X$; it exists because the preimage of $X$ lies in $E_n$. Then the strict transform $X_l\subset M_l$ of $X$ (i.e. the schematic closure of $X\times_MM_l\setminus E_l$) is non-empty, lies in $V_l$ and contains an irreducible component of $V_l$. Since $V_l$ is smooth, $X^\red_l$ is an irreducible component of $V_l$, in particular, it is smooth. So far, we repeated the standard argument, which constructs the resolution $X_l^\red\to X^\red$ of the integral scheme $X^\red$, but we claim that, moreover, $X_l\to X$ is a resolution of $X$.

Locally at the generic point $\eta\in X$ the principalization works as follows: there is no boundary and whenever $\dim(M)>\dim(X)+1$ the order is one by Lemma~\ref{thickmanlem}, so one simply replaces $M$ by a maximal contact -- any smooth divisor containing $X$. Once $\dim(M)=\dim(X)+1$, locally at $\eta$ one has that $\calI_X=(t^n)$, where $n$ is the thickness of $X$ and $H=V(t)$ is the reduction of $X$. It follows that the order is $n$, the only maximal contact is $H$ and the algorithm simply resolves the marked ideal $(\calI_X,n)$ by blowing up $H$. So, generically on $X$ there are $r=\dim(M)-\dim(X)$ non-trivial steps: for the first $r-1$ steps the algorithm reduces the dimension of $M$ and then blows up $H=\tilX$. By our assumption the latter blowing up number is $l+1$ (typically, the principalization of $X$ also involves steps which restrict to empty blowings up in a neighborhood of $\eta$ and are not counted in the above description). We will prove by computing the orders that at this stage the whole $X$ is a ptm, that is, $X_l$ is a ptm.

Looking locally at $\eta$ we see that this blowing up happens when there is no new boundary (we are in Step 1b of the algorithm in \cite{Wlodarczyk}), $\dim(M_l)=\dim(X_l)+1$ and the order of $\calI_{X_l}$ at $x$ is $n$ because the principalization center at a no boundary step is supported on the locus of maximal order and hence the order of $\calI_{X_l}$ equals $n$ everywhere on $H$. Note that $nH$ is the schematic closure of the generic point of ${X_l}$, hence $nH\into {X_l}$ and $\calI_{X_l}$ is contained in the divisorial ideal $\calI_H^n$. In particular, $\calI_H^{-n}\calI_{X_l}$ is an ideal, and since the degrees of $\calI_H^n$ and $\calI_{X_l}$ are equal to $n$ everywhere along $H$, the order of $\calI_H^{-n}\calI_{X_l}$ is zero, that is, $X_l=nH$. So, as claimed $X_l$ is a principally thick manifold and $X_l\to X$ is a projective desingularization.

For the sake of completeness, we note that the blowing up $X_{l+1}\to X_l$ is the last blowing up of the principalization, it happens at Step 1ba in \cite{Wlodarczyk}, and the invariant at this step is $(1,0,1,0\.1,0,n,0,\infty)$.

Step 2. {\it Assume that $X$ is irreducible.} In this case one locally embeds $X$ into a smooth variety $M$ and takes the resolution $X_l\to X$ provided by Step 1. A standard argument used in all works on functorial resolution shows that this is independent of the local choice of embeddings and globalizes to a smooth-functorial method to resolve $X$ -- usually one works with an integral $X$ because principalization does not provide resolution of general non-reduced varieties, but this reduction only uses that $X$ is irreducible.

Step 3. {\it The general case.} In general, let $\eta_1\.\eta_s$ be the generic points of $X$ with the induced scheme structure and let $X_i$ be the schematic closure of $\eta_i$. Then $\coprod_iX_i\to X$ is a modification, which depends on $X$ in a way which is easily seen to be compatible with smooth morphisms $X'\to X$. If $Y_i\to X_i$ are the resolutions provided by Step 2, then $\calF(X)=\coprod_iY_i$ is a resolution of $X$, and by the above $\calF$ is smooth-functorial.
\end{proof}

\section{Resolution on ptms}\label{onptm}
In this section we will prove our main result -- a version of semistable reduction over closed subschemes of a trait (i.e. zero dimensional schemes with principal radical). In the sequel, given a ptm $X$ we will use the notation $\tilX=X^\red$ to denote its reduction.

\subsection{Principalization on ptms}\label{princsec}
In this section we construct principalization of ideals on a ptm using the usual principalization on its reduction.

\subsubsection{Blowings up}
Principalization is achieved by iteratively blowing up regular centers.

\begin{lem}\label{blowuplem}
Assume that $X$ is a ptm and $V\into X$ a regular center. Then $X'=\Bl_V(X)$ is a ptm too, and the reduction of $X'$ is the blowing up of the reduction of $X$ along $V$, that is, $\tilX'=\Bl_V(\tilX)$.
\end{lem}
\begin{proof}
The second claim follows from the fact that $\tilX'$ is the strict transform of the closed subscheme $\tilX$ of $X$, and hence by the classical properties of blowings up $\tilX'$ is the blowing up of $\tilX$ along $V\times_X\tilX=V$. In particular, $\tilX'$ is regular.

It remains to check that $X'$ is a ptm, and this is a local claim over a point $x\in V$. So, we can assume that $X=\Spec(A)$ is a local scheme with closed point $x$. Choose local parameters $\veps,t_1\.t_n$ at $x$, such that $V=V(\veps,t_1\.t_r)$. Then the $\veps$-chart is empty, so $X'$ is covered by the charts $X'_i=\Spec(A_i)$, with $A_i=A[\veps/t_i,t_1/t_i\.t_r/t_i]\subseteq A_{t_i}$, and it is easy to see that $\veps'=\veps/t_i$ generates the nilradical of $A_i$. Thus, the nilradical of $X'$ is locally principal and the reduction $\tilX'$ is regular, that is, $X'$ is a ptm.
\end{proof}

\subsubsection{Snc divisors}\label{boundary}
Working on a ptm $X$ one has to distinguish Weil and Cartier divisors. Weil divisors are just divisors on the reduction, while Cartier divisors contain infinitesimal information as well. Saying a ``divisor'' on $X$ we will always mean an effective Cartier divisor $D=V(f)\subset X$. A divisor $D$ will be called {\em snc} if locally on $X$ it factors as $D=\sum_i D_i$ such that each reduction $\tilD_i=D_i\times_X\tilX$ is smooth and $\tilD=\cup_i\tilD_i$ is an snc divisor on $\tilX$. A divisor $D'$ is $D$-monomial if it is of the form $D'=\sum_i n_iD_i$. A regular closed subscheme $V\into X$ has {\em simple normal crossings with} $D$ if it has simple normal crossings with $\tilD$ as a subscheme of $\tilX$.

\begin{rem}\label{sncrem}
If $D$ is a divisor with an snc reduction $\tilD=\cup_iD_i$, then it can freely happen that $D$ is not snc because the factorization of $\tilD$, does not lift to a factorization of $D$, and the same is true for monomiality. For example, if $X=\Spec(k[x,y,\veps]/(\veps^2))$ then $D=V(xy+\veps)$ is not snc, though $\tilD$ is snc, and $D'=V(y^2+\veps)$ is not $V(y)$-monomial, though $\tilD'$ is $V(y)$-monomial. In addition, there exists different snc divisors with the same reduction, for example, $(x)$ and $(x+\veps)$.
\end{rem}

In the sequel, by a {\em ptm with a boundary} we mean a ptm $X$ and an snc divisor $E\into X$, also called a {\em boundary}. A boundary $E$ is called {\em ordered} if it is provided with a decomposition $E=\sum_{i=1}^lE_i$ into components (not necessarily connected) with smooth reductions $\tilE_i$.

\subsubsection{Addmissible blowings up and transforms}
Assume that $(X,E)$ is a ptm with a boundary and $\calI\subseteq\calO_X$ is an ideal corresponding to a closed subscheme $Z=V(\calI)$, then a {\em $Z$-admissible blowing up} is a blowing up $f\:X'=\Bl_V(X)\to X$, where $V$ is regular, has snc with $E$ and is contained in $Z$. Since $\calI\subseteq\calI_V$ and the exceptional divisor $E_f=V\times_XX'$ is a Cartier divisor, the {\em principal transform} $\calI'=(\calI\calO_{X'})\calI_{E_f}^{-1}$ is defined. It corresponds to the subscheme $Z'$ of $X'$ that we also denote $Z\times_XX'-E_f$ and call the {\em principal transform} of $Z$ under $f$.

In addition, if $E=\sum_{i=1}^l E_i$ is ordered, then for each $i$ we set $E'_i=E_i\times_XX'-V_i\times_XX'$, where $V_i$ is the union of all components of $V$ contained in $E_i$ (and so $V_i\times_XX'$ is the component of $E_f$ contained in $E_i\times_XX'$). Finally, we set $E'_{l+1}=E_f$ and define the total transform of $E$ to be the ordered boundary $E'=\sum_{i=1}^{l+1}E'_i$. Since $E'$ is independent of the ordering of $E$ and an ordering can be chosen locally, we also obtain a definition of total transform for unordered boundaries.

When the thickness is 1 these definitions are nothing but the classical ones from the principalization theory on regular schemes. Moreover, they are compatible with the reduction:

\begin{lem}\label{transformlem}
Assume that $(X,E)$ is a ptm with an (ordered) boundary, $Z\into X$ is a closed subscheme and $f\:X'\to X$ is a $Z$-admissible blowing up with center $V\into Z$, and let $Z'$ and $E'$ be the transforms. Also, let $\tilZ=Z\times_X\tilX$ and $\tilE=E\times_X\tilX$, and let $\tilZ'$ and $\tilE'$ be their transforms under the blowing up of $\tilX$ along $V$. Then, $E'$ is an snc divisor in $X'$, $\tilZ'=Z'\times_{X'}\tilX'$ and $\tilE'=E'\times_{X'}\tilX'$.
\end{lem}
\begin{proof}
Recall that by Lemma \ref{blowuplem} $\tilf\:\tilX'\to\tilX$ is the blowing up along $V$. In particular, $E_f=V\times_XX'$ and $E_\tilf=V\times_\tilX\tilX'$ and it follows that $E_\tilf=V\times_X\tilX'=E_f\times_{X'}\tilX'$. Using that $(Z\times_XX')\times_{X'}\tilX'=Z\times_X\tilX'=\tilZ\times_{\tilX}\tilX'$, we obtain that also $Z'\times_{X'}\tilX'=\tilZ'$.

For concreteness assume that $E=\sum_{i=1}^nE_i$ is ordered and let $E'=\sum_{i=1}^{n+1}E'_i$ and $\tilE'=\sum_{i=1}^{n+1}\tilE'_i$. The above paragraph easily implies that $\tilE'_i=\tilE_i\times_{X'}\tilX'$ for any $i$. Since by the usual theory $\tilE'$ is an snc divisor with irreducible components $\tilE'_i$, it follows that $E'$ is snc.
\end{proof}

\subsubsection{Principalization}
Assume that $(X,E)$ is a ptm with a boundary and $Z\into X$ is a closed subscheme, then a {\em principalization} of $Z$ (or of its associated ideal $\calI_Z$) on $(X,E)$ is a sequence of $Z_i$-admissible blowings up $X_{i+1}=\Bl_{X_i}(V_i)\stackrel{f_i}\to X_i$ with $0\le i\le n-1$, where $(X_0,E_0,Z_0)=(X,E,Z)$ and for $0\le i<n$ one defines inductively $E_{i+1}$ and $Z_{i+1}$ to be the transforms of $E_i$ and $Z_i$, and $Z_n=\emptyset$. We will use dashed arrows to denote sequences of morphisms, e.g. $X_n\dashto X$.

\subsubsection{The principalization theorem}
It turns out that the usual principalization on $\tilX$ induces a principalization on $X$ yielding the following result:

\begin{theor}\label{princth}
There exists a method which to any triple $(X,E,Z)$ consisting of a ptm $X$ of finite type over a field $k$ of characteristic zero, a boundary $E$ and a closed subscheme $Z\into X$ associates a principalization $$\calP(X,E,Z)\:X_n\dashto X_0=X$$ so that $\calP$ depends on the triple $(X,E,Z)$ smooth-functorially: if $X'\to X$ is a smooth morphism of ptms, $E'=E\times_XX'$ and $Z'=Z\times_XX'$, then $\calP(X',E',Z')$ is obtained by pulling back the blowings up sequence $\calP(X,E,Z)$ and removing all trivial blowings up (those with empty centers).
\end{theor}
\begin{proof}
Choose any smooth-functorial principalization method for varieties, e.g. the method from \cite{Wlodarczyk}. Set $\tilE=E\times_X\tilX$ and $\tilZ=Z\times_X\tilX$ and let $\tilX_{i+1}=\Bl_{V_i}(\tilX_i)$, $0\le i<n$ be the principalization of $\tilZ$ on $(\tilX,\tilE)$. We pushforward the blowing up sequence $\tilX_n\dashto\tilX_0=\tilX$ by setting $X_0=X$ and $X_{i+1}=\Bl_{V_i}(X_i)$. This makes sense because applying Lemma~\ref{blowuplem} inductively we obtain that $\tilX_i$ is, indeed, the reduction of $X_i$ and hence $V_i\into\tilX_i\into X_i$. Applying Lemma~\ref{transformlem} inductively we obtain that $V_i$ has snc with the boundary $E_i$, which is the transform of $E_{i-1}$, and $V_i$ is a subscheme of $Z_i$, which is the transform of $Z_{i-1}$. In other words, the blowings up are $Z_i$-admissible. Finally, Lemma~\ref{transformlem} also implies that $\tilZ_i=Z_i\times_{X_i}\tilX_i$ for any $i\le n$. Since $\tilZ_n$ is empty this implies that $Z_n$ is empty, and hence the sequence $X_n\dashto X_0=X$ is a principalization of $Z$, that we denote $\calF(X,E,Z)$.

Smooth-functoriality of the constructed principalization follows from the smooth-functoriality of the usual principalization and the easily checked fact that all other ingredients, such as the reduction morphism $\tilX\to X$, are compatible with smooth base changes.
\end{proof}

\subsection{Divisors with monomial reduction}\label{monomsec}

\subsubsection{Modifications with a trivial reduction}
Let $h\:Y\to X$ be a morphism of ptms. Then the following conditions are equivalent: (i) $h$ is a bijective modification, (ii) is a modification and a homeomorphism, (iii) $h$ is generically an isomorphism and the reduction $\tilh$ is an isomorphism. If these conditions are satisfied, we say that $h$ is a {\em modification with a trivial reduction}.

\begin{exam}
An important example of a modification with a trivial reduction is the blowing up $Y=\Bl_\tilD(X)$ along a divisor $\tilD\into\tilX$ of the reduction. Locally on $X$ we can assume that $X=\Spec(A)$ is affine and $f\in A$ is such that its image $\tilf\in A/\calN$ defines $\tilD$, where $\calN$ is the nilradical. Then $\tilD=V(\calN+(f))$ in $X$ and $Y=\Spec(A[\frac{\calN}{f}])$ consists of the $f$-chart, since other charts correspond to nilpotent elements and hence are empty.
\end{exam}

More generally, if $\tilD_1\.\tilD_l$ are divisors of $\tilX$ then by induction on $l$ one can define a sequence $X_l\dashto X_0=X$ with $X_{i+1}=\Bl_{\tilD_i}(X_i)$ because $\tilD_i\subset\tilX=\tilX_i$.

\subsubsection{Monomialization}
We saw in Remark \ref{sncrem} that if $D$ is a divisor with an $\tilE$-monomial reduction, $D$ itself does not have to be monomial with respect to any boundary or it can be monomial with respect to another boundary. All these problems are remedied by applying a simple blowing up sequence.

\begin{lem}\label{monomiallem}
Assume that $X$ is a ptm with an ordered boundary $E$ and $D$ is a Cartier divisor such that $\tilD$ is $\tilE$-monomial. Let $\tilD=\sum_{i,n}n\tilD_{i,n}$ be the decomposition, where $\tilD_{i,n}$ is the union of components of $D$ of multiplicity $n$ that lie in $\tilE_i$. Let $X'\dashto X$ be the modification with trivial reduction obtained by blowing up centers $V_{1,1}\.V_{1,n_1},V_{2,1},\dots$, where $V_{i,j}=\coprod_{n\ge j}\tilD_{i,n}$. Then $D\times_XX'$ is monomial with respect to the transform $E'$ of $E$.
\end{lem}
\begin{proof}
It suffices to check the claim locally at a point $x\in X$. Let $n_i\ge 0$ be such that $x\in \tilD_{i,n_i}$. Choose parameters $\veps,t_1\. t_l$ so that $E=V(t_1\dots t_r)$ at $x$. Then locally at $x$ we have that $\tilD=V(\prod_{i=1}^r\tilt_i^{n_i})$ and hence $D=V(f)$ for $f=\prod_{i=1}^rt_i^{n_i}+\veps a$. Since $\veps'=\veps/\prod_{i=1}^rt_i^{n_i}$ and $t_1\.t_l$ are parameters of $X'$ at $x$ we obtain that $f=\prod_{i=1}^rt_i^{n_i}(1+\veps' a)$ in $\calO_{X',x}$ and $1+\veps' a$ is a unit. This finishes the proof.
\end{proof}

For the sake of completeness we note that the same sequence is produced by the usual principalization:

\begin{rem}
In fact, the sequence $\tilX'\dashto\tilX$ in the above lemma is precisely the principalization blowings up sequence of $\tilD$ on $(\tilX,\tilE)$ produced by the classical order reduction algorithms with marking 1. So, similarly to the argument in the proof of Theorem~\ref{princth}, its pushforward $X'\dashto X$ in the lemma is a principalization of $D$.
\end{rem}

\subsection{Resolution of $B$-schemes}\label{Bschemes}

\subsubsection{Distinguished $B$-pairs}\label{Bpairs}
Let $B=\Spec(k[\pi]/(\pi^n))$. By a {\em distinguished $B$-pair} we mean a morphism $f\:X\to B$ and a divisor $E$ on $X$ such that $X$ is a ptm, $E$ is a boundary on $X$, and $\pi\calO_{X}=\calN\calI$, where $\calN$ is the nilradical of $\calO_{X}$ and $V(\calI)$ is an $E$-monomial divisor. In down to earth terms this means that locally on $X$ there exists coordinates $(\veps,t_1\.t_m)$ such that $E=V(t_1\dots t_r)$ and $f^\#(\pi)=\veps t_1^{d_1}\dots t_r^{d_r}$. In particular, $X\to B$ locally factors through a morphism $$h\:X\to X_0=\Spec(k[\veps,t_1\.t_m]/(\veps^n)),$$ and if $X\to B$ is of finite type, then $h$ is \'etale.

\begin{rem}\label{Bpairsrem}
Note that $X_0$ is a monomial irreducible component of the log smooth $B$-scheme $$Y_0=\Spec(k[\veps,t_1\.t_m]/(\veps^n t_1^{nd_1}\dots t_r^{nd_r}))$$ given by $\veps^n=0$. If $(X,E)$ is a distinguished $B$-pair of finite type, then $h\:X\to X_0$ is \'etale and hence locally on $X$ it can be extended to an \'etale morphism $Y\to Y_0$ such that $h$ is its pullback. In particular, locally $X$ embeds as a monomial irreducible component into a log smooth $B$-scheme $Y$.
\end{rem}

\subsubsection{The main result}
Now we can prove our main result on resolution of $B$-schemes.

\begin{theor}\label{mainth1}
Let $k$ be a field of characteristic zero and $B=\Spec(k[\pi]/(\pi^n))$. Then there exists a procedure $\calF$ which obtains as an input a generically smooth morphism of finite type $X\to B$ and a closed nowhere dense subscheme $Z\into X$, and outputs a modification $\calF(X,Z)\:X'\to X$ with $X'$ a ptm and a boundary $E'$ on $X'$ such that $Z'=Z\times_XX'$ is an $E'$-monomial divisor and $(X',E')$ is a distinguished $B$-pair. In addition, $\calF$ depends smooth functorially on the input: if $Y\to X$ is smooth and $T=Z\times_XY$, then $\calF(Y,T)=\calF(X,Z)\times_XY$.
\end{theor}
\begin{proof}
First, by Theorem~\ref{nonembeddedth} there exists a modification $X_1\to X$ with $X_1$ a ptm. Second, applying Theorem~\ref{princth} to $X_1$ and $Z\times_XX_1$ we obtain a modification of ptms $X_2\to X_1$ and an ordered boundary $E_2$ on $X_2$ such that $Z_2=Z\times_XX_2$ is $E_2$-monomial. Third, the generic smoothness of $f$ implies that generically $\pi$ generates the nilradical $\calN_2\subset\calO_{X_2}$ on an open subscheme $U$. Principalizing $X_2\setminus U$ on $(X_2,E_2)$ by Theorem~\ref{princth} we obtain a modification $X_3\to X_2$, such that $X_3$ is provided with an ordered boundary $E_3$ and $\calN_3|_U=(\pi)|_U$ for $U=X_3\setminus E_3$.

Now, locally on $X_3$ there exists a factorization $(\pi)=\calN_3\calI$, and the locally principal ideal $\calI$ is invertible because it is trivial at the generic point and $X_3$ has no embedded components. Note that $\tilcalI=\calI\calO_{\tilX_3}$ is an invertible ideal supported on the snc divisor $\tilE_3=E_3\times_{X_3}\tilX_3$ and hence is $\tilE_3$-monomial, but $\calI$ does not have to be $E_3$-monomial. Even worse, it is not uniquely defined hence the factorization does not have to globalize. However, blowing up the components of $E_3$ as in Lemma~\ref{monomiallem} we obtain a modification $X'\to X_3$ with a trivial reduction such that the pullback of each $\calI$ as above becomes monomial, that is, on $X'$ we locally have factorizations $(\pi)=\calN'\calI$ with an $E'$-monomial $\calI$. Such a factorization is unique because if $(\pi)=\calN'\calI'$ is another factorization with an $E_3$-monomial $\calI'$, then $V(\calI)$ and $V(\calI')$ restrict to the same divisor on $\tilX_3$ and hence  contain the same components of $E_3$ with the same multiplicities. Thus the factorization $(\pi)=\calN\calI$ globalizes, and this shows that $X'\to X$ is as required.

The smooth-functoriality of the construction follows from functoriality of all ingredients we used.
\end{proof}

\subsection{Other geometric spaces}\label{othersec}
We deduced our results from the classical principalization theorem for algebraic varieties. Now, let us discuss what can be done for more general schemes, complex analytic spaces, etc. These are nowadays standard reductions that were described in details in a few papers, so we only outline them.

\subsubsection{Analytic spaces}\label{spacesec}
It is folklore knowledge that the classical principalization theorem and its proof apply almost verbatim for complex analytic spaces and non-archimedean analytic spaces of characteristic zero, though, to the best of our knowledge, the only formalized framework, in which this was checked, is given in \cite{Bierstone-Milman}. Starting with principalization for analytic spaces one can establish analogues of Theorems~\ref{nonembeddedth}, \ref{princth} and \ref{mainth} using precisely the same arguments as in the paper.

\subsubsection{Schemes with enough derivations}
In fact, the classical principalization algorithm and its justification applies to the wider context of ideals on regular schemes with enough derivations, where the latter were defined in \cite[Remark~1.3.2(iii)]{non-embedded}. The only new (and simple) thing one should check is that any smooth blow up of an $X$ with enough derivations also has enough derivations. Strictly speaking, this was also never published, but a much more involved case of an algorithm in the setting of enough derivations was worked out with all details in \cite{ATW}. Moreover, the functoriality holds with respect to arbitrary regular morphisms and not only the smooth ones. Once principalization for such schemes is known, our arguments extend to this setting verbatim, and one obtains that Theorems~\ref{nonembeddedth}, \ref{princth} and \ref{mainth} hold for schemes locally embeddable into regular schemes with enough derivations, and the functoriality is with respect to arbitrary regular morphisms. Then the usual formal argument (see \cite[\S5.2]{non-embedded}) implies that the same results extend to analytic spaces (recovering the claim of \S\ref{spacesec}) and formal schemes locally embeddable into regular formal schemes with enough derivations.

\subsubsection{Arbitrary qe schemes}
It seems certain that the maximal generality, in which Theorems~\ref{nonembeddedth}, \ref{princth} and \ref{mainth} should hold, is that of arbitrary qe schemes of characteristic zero. However, resolution of non-reduced schemes was not established in such generality and only a weaker form of principalization was proved in \cite[Theorem~1.1.11]{Temkin-embedded} -- it outputs a regular scheme with a monomial ideal, but the intermediate steps can be singular. This would suffice to deduce the last two theorems, but it cannot be used for the non-embedded resolution of generically nil-principal schemes. All in all, the only stumbling block is resolution of arbitrary non-reduced qe schemes.

\section{Complements}\label{mainsec}
In this section we will slightly refine our main result by constructing a global embedding of a modification into a log smooth $B$-scheme. This requires an additional preparation.

\subsection{Factorization}\label{factorsec}
In addition to the analogues of the usual resolution and principalization we will need certain results about factorization and resolution of indeterminacies for ptms. We will only cover the cases that will be needed in our applications. For shortness, by a {\em component} of a divisor $E$ we mean a union of few its connected components.

\begin{theor}\label{factorth}
Assume that $f\:Y\to X$ is a modification of ptms with a trivial reduction. Assume, in addition, that there exists an ordered boundary $E=\sum_{i=1}^lE_i$ on $X$ such that $f$ is an isomorphism over $X\setminus E$. Then there exists a sequence of blowings up of components of $\tilE_i$'s $Y'\dashto Y$ such that the composed morphism $f'\:Y'\to X$ splits into a composition of blowings up of components of $\tilE_i$'s. The sequence $Y'\dashto Y$ and the splitting of $f'$ depend on the datum functorially with respect to smooth morphisms $\oX\to X$ and pullbacks $\of\:\oY=Y\times_X\oX\to\oX$ and $\oE=E\times_X\oX$.
\end{theor}
\begin{proof}
To simplify notation we identify the underlying topological spaces of $X$ and $Y$. Notice that locally we can factor $\calN_X\calO_Y$ as $\calI\calN_Y$, where $\calI$ is invertible -- just take local generators $\veps_Y$ and $\veps_X$, then $\veps_X=a\veps_Y$ and we take $\calI=(a)$. This does not yield a global divisor $\calN_Y/\calN_X$ because an element $a$ is defined only up to adding an element annihilated by $\veps_Y$, but the reduction $(\tila)$ is unique and we obtain a global divisor on $\tilX$ that we denote $\tilD=\wt{\calN_Y/\calN_X}$. Clearly, it is $\tilE$-monomial, so we fix the decomposition $\tilD=\sum_{i,n}n\tilD_{i,n}$ with $\tilD_{i,n}\subset\tilE_i$. Now, let $Y'\dashto Y$ be the sequence obtained by successive blowing up the centers $V_{i,j}=\coprod_{n\ge j}\tilD_{i,n}$ and let $E'$ be the transform of the boundary to $Y'$. By Lemma~\ref{monomiallem} the local divisors $(a)$ become $E'$-monomial on $Y'$. The reductions $(\tila)$ agree on the intersections (this happens already on $X$), hence we obtain a factorization $\calN_{Y'}=D'\calN_X$ with an $E'$-monomial divisor $D'$.

Choose the presentation $D'=\sum_{i,n}nD'_{i,n}$ with $D'_{i,n}\subset E_i$ (we do not need this, but in fact only even $n$'s show up and $\tilD'_{i,2n}=\tilD_{i,n}$). We claim that $Y'\to X$ splits into the blowing up sequence $Y'=X_m\dashto X_0=X$ with centers $V'_{i,j}=\coprod_{n\ge j}\tilD'_{i,n}$. Indeed, this can be checked locally at a point $x\in X$. Choose parameters $\veps_X,t_1\.t_n\in\calO_{X,x}$ such that $E=V(t_1\dots t_r)$ and let $n_i$ be such that $x\in D'_{i,n_i}$. Then $\veps_Y=\veps_X/\prod_{i=1}^rt_i^{n_i}$ is a generator of the nilradical of $\calO_{Y',x}$ and hence an embedding $\calO_{X_m,x}\subseteq\calO_{Y',x}$ arises. Moreover, it is an isomorphism by Nakayama's lemma because the reductions coincide and $\calO_{X_m,x}$ contains a generator of the nilradical of $\calO_{y',x}$. Thus, $Y'=X_m$ and we have constructed a splitting of $Y'\to X$. The check that our construction is smooth-functorial is straightforward and we omit it.
\end{proof}

\begin{cor}\label{morphismCor}
	Let $X\rightarrow B_n=\Spec(k[\pi]/(\pi^n))$ be generically smooth. Then there exists a construction $\mathcal{G}(X\rightarrow B_n)=Y\rightarrow X\rightarrow B_n$ such that the composed morphism is smooth away from an snc divisor. This construction is functorial with respect to smooth morphism.
\end{cor}
\begin{proof}
	As $X\rightarrow B_n$ is generically smooth, $X$ is generically nilprincipal and we may apply Theorem \ref{nonembeddedth} to get $X'\rightarrow X$ such that $X'$ is a ptm. Let $Z\subset X'$ be the singular locus of the composed morphism $X'\rightarrow X\rightarrow B_n$, and by Theorem \ref{princth} we have a principalization $\mathcal{P}(X', \emptyset, Z):Y\rightarrow X'$ such that the composed morphism $Y\rightarrow B_n$ is smooth away from an snc divisor. This process is smooth-functorial as both ingredients are.
\end{proof}
\begin{rem}\label{etaleloc}
	We make two remarks on this result:
	\begin{enumerate}
		\item The morphism $Y\rightarrow X'$ is a modification of ptms with trivial reduction, so if we wish to apply Theorem \ref{factorth} we may.
		\item In practice, this means that given a generically smooth $X\rightarrow B_n$ there is a resolution which \'etale-locally looks like $X'=\text{Spec}k[\epsilon, t_1,..., t_l]/(\epsilon^n)\rightarrow B_n$ given by $\pi\mapsto\epsilon\cdot t_1^{n_1}...t_l^{n_l}$.
	\end{enumerate}
\end{rem}

\subsection{Retracts}\label{retractsec}
By a retract of a ptm $X$ we mean a morphism $r\:X\to\tilX$ which is a retract of the closed immersion $i\:\tilX\into X$, that is, $r\circ i=\id_\tilX$. In general, a retract does not exist, but if $X$ is of finite type over a field $k$, then it exists whenever $X$ is affine. Indeed, $X$ is a nilpotent thickening of $\tilX$ and $\tilX$ is smooth over $k$, hence the $k$-morphism $\id_\tilX$ lifts to a $k$-morphism $r\:X\to\tilX$. (For an arbitrary ptm $X$, a retract at least exists at the generic points $\eta\in X$, namely, $r_\eta\:\eta\to\tileta$ corresponds to a choice of a field of definition $k(\tileta)\into\calO_{X,\eta}$ of the Artin ring $\calO_{X,\eta}$.) We will show that after blowing up $X$ one can extend a generic retract to a retract of the whole $X$. The argument will be similar to the proof of Theorem~\ref{factorth} but a bit more elaborate. Here is the main particular case we will need.

\begin{lem}\label{retlem}
Let $X$ be a ptm of finite type over a field $k$ with nilradical $\calN\subset\calO_X$, let $D=V(\calI)$ be a Cartier divisor with a smooth reduction $\tilD=D\times_X\tilX$ and complement $X_0=X\setminus D$, and let $i\:X_0\into X$ denote the open immersion. Assume that $D$ is irreducible with generic point $\nu$ and $r_0\:X_0\to\tilX_0$ is a retract of $X_0$, then

(i) The retract $r_0$ extends to $r$ if and only if the homomorphism $\phi\:i_*\calO_{\tilX_0}\into i_*\calO_{X_0}$ induced by $r_0$ takes $\calO_\tilX$ to $\calO_X$.

(ii) The following three conditions are equivalent and they hold for a large enough $n$:
(a) $\phi$ takes $\calO_\tilX$ to $\calO_X[\calI^{-n}\calN]$,
(b) $\phi$ takes $\calO_{\tilX,\nu}$ to $\calO_{X,\nu}[\calI_\nu^{-n}\calN_\nu]$,
(c) $r_0$ extends to a global retract on $\Bl_{\widetilde{nD}}(X)$.

(iii) Assume that $n\ge 1$ is such that the conditions in (ii) are satisfied and let $X'=\Bl_\tilD(X)$. Then $\phi$ takes $\calO_{\tilX'}$ to $\calI^{1-n}\calO_{X'}$.
\end{lem}
\begin{proof}
All claims can be checked locally on $X$, where we use in (ii) that $X$ is quasi-compact. Thus, we can assume that $X=\Spec(A)$ and $\calI=(t)$, and then $X_0=\Spec(A_t)$, $\tilX=\Spec(\tilA)$ and $\tilX_0=\Spec(\tilA_t)$, where $\tilA=A/\calN$. Let $s_0\:\tilA_t\to A_t$ be the section corresponding to $r_0$. Since $\tilX\into X$ is a homeomorphism, $r_0$ extends to $r$ if and only if the homomorphism of sheaves extends, that is $s_0(\tilA)\subseteq A$. This proves (i).

Recall that $X'=\Bl_{\widetilde{nD}}(X)=\Bl_{(t^n)+\calN}(X)$ consists of the $t$-chart, hence $X'=\Spec(A_n)$, where $A_n=A[t^{-n}\calN]$. Clearly, $X'\to X$ is a homeomorphism and we denote the preimage of $\nu$ by $\nu'$. Applying (i) to $X'$ we obtain that (a) and (c) are equivalent, and the equivalence of (a) and (b) reduces to the claim that if $r_0$ extends to $\nu'$, then $r_0$ extends to the whole $X'$. But $X'$ can be obtained from the ptm $X$ by $n$ successive blowings up of $\tilD$, hence $X'$ is a ptm and any regular function on $X_0$ that extends to $\nu'$ automatically extends to the whole $X'$ by Corollary~\ref{CMcor}. The claim about extension of $r_0$ follows.

Let us now prove that for a large $n$ the condition (a) is satisfied. Given $\tila\in\tilA$ choose a lift $a\in A$ and notice that $s_0(\tila)-a\in\calN_t$ and hence lies in a subring $A_n=A[t^{-n}\calN]$ of $A_t$. Choosing a set $\tila_1\.\tila_l$ of $k$-generators of $\tilA$ and a large enough $n$, we obtain that $s_0(\tila_i)\in A_n$ for any $1\le i\le l$ and hence $s_0(\tilA)\subset A_n$. This proves the second part of (ii).

Finally, in (iii) $X'=\Spec(A')$, where $A'=A_1$, and hence the nilradical of $A'$ is $\calN'=t^{-1}\calN$ and $\tilA=\tilA'$. Therefore, $s_0(\tilA')\subset A_n=A'[t^{1-n}\calN']$, proving (iii).
\end{proof}

\begin{theor}\label{retractth}
There exists a method which to any ptm with a boundary $(X,E)$ of finite type over a field $k$ of characteristic zero and a generic retract $r_\eta\:\eta\to\tileta$, where $\eta=\coprod_{i=i}^n\Spec(\calO_{X,\eta_i})$ and $\eta_1\.\eta_n$ is the set of generic points of $X$, associates a sequence of admissible blowings up $$\calR(X,E,r_\eta)\:(X',E')\dashto (X,E)$$ such that $r_\eta$ extends to a retract $r\:X'\to\tilX'$. Moreover, one can construct $\calR$ functorially with respect to \'etale morphisms $Y\to X$ provided with a compatible pair of generic retracts $r_{\eta_Y}$ and $r_{\eta_X}$.
\end{theor}
\begin{proof}
We will first construct the retract as a composition of two admissible blowings up sequences, and only then discuss the functoriality. To simplify notation we assume that $X$ is irreducible and $\eta$ is the generic point, but the argument works in general as well. Let $s_\eta\:k(\eta)\into\calO_{X,\eta}$ be the field of coefficients corresponding to $r_\eta$. Recall that $X$ has no embedded components, hence $\calO_{X,x}\subseteq\calO_{X,\eta}$ for any $x\in X$ and by Lemma~\ref{retlem}(i) $r_\eta$ extends to $r$ locally at $x\in X$ if and only if $s_\eta(\calO_{\tilX,x})\subseteq\calO_{X,x}$. Clearly, this is an open condition and we denote by $U$ the open set of such points $x\in X$, and denote by $Z=X\setminus U$ the complement. In other words, $Z$ is the indeterminacy locus of the rational map $r_\eta\:X\to\tilX$. The first part of $\calR$ is the principalization $\calP(X,E,Z)\:X'=X_l\dashto X_0=X$ of $Z$. Let $E'_1\.E'_l$ be the components of $E'$ created by $\calP$, then the indeterminacy locus $Z'$ of $r_\eta\:X'\to\tilX'$ is contained in $E''=\sum_{i=1}^lE'_i$.

The remaining procedure iteratively blows up unions of components of $\tilE''$. In particular, the reduction $\tilX'$ remains unchanged, so we can identify $\tilE'_i$ with subschemes of further blowings up. The choice of the center on each step is done by the following rule: for any generic point $\nu\in E''$ let $i=i(\nu)$ be the number for which $\nu\in E'_i$ and let $n(\nu)$ be the minimal $n\ge 0$ which satisfies the condition of Lemma~\ref{retlem}(ii). The step takes all generic points $\nu$ such that $n(\nu)>0$ (i.e. the point is in the indeterminacy locus) and the lexicographic value of $(n(\nu),i(\nu))$ is maximal, and blows up the reduced closed subschemes with these generic points. In other words, we blow up components of $\tilE'_i$ with the maximal value of the invariant $(n(\nu),i(\nu))$. The maximal value of the invariant drops after each blowing up by Lemma~\ref{retlem}(ii) and (iii). Part (i) of the same lemma says that once $n$ drops to zero, the indeterminacy locus is empty, that is, $r_\eta$ extends to a global retract.

Finally, let us discuss the functoriality. So let us assume that $f\:Y\to X$ is \'etale, $E_Y=E\times_XY$ and a generic retract $r_{\eta_Y}$ is the pullback of $r_{\eta_X}$. The principalization is functorial with respect to smooth morphisms, so we should only check that the indeterminacy locus and the invariant $n(\nu)$ used in the second step are compatible with $f$. Moreover, if $D$ is the irreducible component of $\nu$, then by Lemma~\ref{retlem}(ii) $n(\nu)$ is the minimal $n$ such that blowing up $\widetilde{nD}$ resolves the indeterminacy at $\nu$. Blowings up are compatible with smooth (even flat) morphisms, hence it suffices to check functoriality of the indeterminacy locus: if $y\in Y$ and $x=f(y)$, then $r_{\eta_X}$ extends to $x$ if and only if $r_{\eta_Y}$ extends to $y$. Furthermore, since the indeterminacy locus is open, localizing $X$ and $Y$ appropriately we can reduce the problem to the following particular case: $r_{X_\eta}$ extends to a global retract $r_X$ if and only if $r_{Y_\eta}$ extends to a global retract $r_Y$. Clearly, it suffices to check the latter locally, so we can assume that $X$ and $Y$ are affine.

Assume that $r_{\eta_X}\:\eta_X\to\eta_\tilX$ extends to $r\:X\to\tilX$. Then a morphism $Y\to X\to\tilX$ arises and the $\tilX$-morphism $\id_\tilY$ lifts uniquely to a $\tilX$-morphism $r_Y\:Y\to\tilY$ by the \'etalenes of $\tilY$ over $\tilX$ and affineness of $Y$. Thus, the retract $r_X$ lifts uniquely to a retract $r_Y$. Since the lift $r_{\eta_Y}$ of $r_{\eta_X}$ is also unique by the same argument, both agree, that is $r_Y$ is the extension of $r_{\eta_Y}$. Conversely, if $r_{\eta_Y}$ extends to $r_Y$, then it is easy to see that $r_{\eta_X}$ extends to $r_X$.

%$$
%\xymatrix{
%Y\ar[r]\ar@{..>}[rd] & X& \\
%\tilY\ar[d]\ar@{^(->}[r]\ar@{^(->}[u]&\tilY\times_\tilX X\ar[d]\ar[r]\ar[u]&\tilY\ar[d]\\
%\tilX\ar@{^(->}[r]&X\ar[r]^{r_X}&\tilX.
%}
%$$

\end{proof}

We do not know if functoriality holds for smooth morphisms as well.

\subsection{Log blowings ups}\label{logblowsec}
We will need one more construction on a ptm with boundary. Informally speaking, it replaces boundary components by varieties intersecting $X$ so that one obtains a thick version of a normal crossings variety. The most convenient way to realize this is by using log geometry.

\subsubsection{Associated log structure}
As often happens one can encode the boundary in a log structure, and in various applications this is even a more conceptual approach. Given a ptm with a boundary $(X,E)$ we define the associated log structure $M_X$ as follows: for a point $x\in X$ let $t_0=\veps$ be a generator of the nilradical of $\calO_x$ and let $t_1\.t_r\in\calO_x$ be parameters that define the irreducible components of $E$ passing through $x$ and consider the prelogarithmic structure $\oM_x=\oplus_{i=0}^r\bbN\log(t_i)\to\calO_x$ which sends $\log(t_i)$ to $t_i$. Since the choice of $\veps$ and $t_i$ is unique up to a unit, the associated log structure $M_x$ is independent of the choices, and it is easy to see that the local definitions patch to a global log structure $M_X$ on $X$. Clearly, $E$ can be reconstructed from $M_X$, so we also address to the log scheme $(X,M_X)$ as a {\em ptm with a boundary}.

\begin{rem}
Note that our log structure is the direct sum of the hollow log structure generated by $\log(\veps)$ and the divisorial log structure generated by $\oplus_{i=1}^r\bbN\log(t_i)$. This is the hollow part which makes the log blowings up in \S\ref{redsec} below so different from the usual blowings up.
\end{rem}

\subsubsection{Log blowings up}
Recall that in log geometry there exists a canonical operation making log ideals invertible. Namely, given a log scheme $(X,M_X)$ and an ideal $J\subseteq M_X$ (giving it is equivalent to giving an ideal $\oJ\subseteq\oM_X$) the {\em log blowing up} of $X$ along $J$ is the universal morphism of log schemes $\LogBl_J(X)=X'\to X$ such that $J'=JM_{X'}$ is invertible. We refer to \cite{Niziol} for details or to \cite[\S4.3]{lognotes} for a brief review.

\begin{rem}\label{logrem}
(i) Unlike the usual blowing up there is no distinction between principal and invertible ideals in $M_X$, and this makes the operation more functorial. In particular, it is compatible with arbitrary base changes (and not only flat or log flat ones).

(ii) In nice cases, for example, when $X$ is log regular, $\LogBl_J(X)=\Bl_{J\calO_X}(X)$ coincides with the usual blowing up. 

(iii) As a corollary from (i) and (ii) one obtains a very simple way to compute log blowings up: just consider a (local) chart $X\to X_0$, then $J$ is the pullback of $J_0\subseteq M_{X_0}$, and we simply have $\LogBl_J(X)=\Bl_{J_0\calO_{X_0}}(X_0)\times_{X_0}X$. This also indicates that in general (e.g. when the log structure is hollow) such a morphism does not have to be birational, can create new components and can even increase the dimension, see \cite[4.3.10]{lognotes} for some examples.
\end{rem}

\subsubsection{Log blowings up of reduced divisors}\label{redsec}
We will use a very special form of log blowings up of a ptm with a boundary $(X,M_X)$.

\begin{lem}\label{logblowlem}
Let $(X,M_X)$ be a ptm with a boundary and $\tilD$ a smooth component of the induced boundary $\tilE$ of $\tilX$. Then $X''=\Bl_\tilD(X)$ is an irreducible component of $X'=\LogBl_\tilD(X)$. In addition, $X''$ is a monomial subscheme of $X'$.
\end{lem}
\begin{proof}
The claim can be checked locally over a point $x\in\tilD$, so assume that $X=\Spec(A)$ and $\veps,t_1\.t_n$ is a family of parameters such that $\tilD=(\veps,t_1)$ and $E=V(t_1\dots t_r)$ at $x$. Then $Y=\Spec(\bbZ[\veps,t_1\.t_r])$ is a chart for $X$, and hence $X'\to X$ is the base change of the log blowing up of $Y$ along $(\veps,t)$, where we set $t=t_1$ for shortness (see Remark~\ref{logrem}(iii)). Since $Y$ is log regular, the latter is the usual blowing up described by the charts, and hence $X'$ is glued from two charts: $X'_t=\Spec(A[\veps']/(t\veps'-\veps))$ and $X'_\veps=\Spec(A[t']/(t'\veps-t))$.

Recall that $X''=\Spec(A'')$, where $A''=A[\frac{\veps}{t}])$, and note that $A''$ is the image of $A'=A[\veps']/(t\veps'-\veps)$ in $A_t$, hence $X''$ is the irreducible component of $X'$, is the strict transform of $X$ and lies in the $t$-chart. Furthermore, it is easy to see that $A''=A'/(\veps'^h)$, where $h$ is the thickness of $X$, hence $X''$ is given by vanishing of the monomial ideal $(\veps'^h)$.
\end{proof}

Here is a concrete typical example.

\begin{exam}
Let $A=k[t,\veps]/(\veps^n)$ and $X=\Spec(A)$ with the log structure induced by $\bbN\log(\veps)\oplus\bbN\log(t)$. Then $\tilD=V(t,\veps)$ is a monomial subscheme, which is also a smooth divisor in $\tilX$. The log blowing up $X'=\LogBl_\tilD(X)$ is glued from two charts: $$X'_t=\Spec(A[\veps']/(t\veps'-\veps))=\Spec(k[\veps',t]/(t^n\veps'^n))$$ and $$X'_\veps=\Spec(A[t']/(t'\veps-t))=\Spec(k[t',\veps]/(\veps^n)).$$ Thus, $X'$ has two irreducible components. The first component is $\Bl_\tilD(X)=\Spec(k[t,\veps']/(\veps'^n))$; it is contained in $X'_t$. The other component is a ptm with nilradical $(\veps)$ and reduction isomorphic to $\bfP^1_k$ with coordinate $t'$; it is contracted to $\tilD$ in $X$.
\end{exam}

For our applications we will need to construct sequences of log blowings up. The following results follows from Lemma~\ref{logblowlem} by an obvious induction on the length of the sequence.

\begin{cor}\label{logblowcor}
Let $(X,M_X)$ be a log ptm with a boundary, and let $\tilD_1\.\tilD_l$ be smooth components of the associated snc boundary $\tilE\subset\tilX$. Then there exists a sequence of log blowings up $X_l\dashto X_0=(X,M_X)$ such that the following conditions are satisfied: the strict transform $X'_i$ of $X$ in $X_i$ is a monomial subscheme of $X_i$, each $X'_{i+1}$ is the blowing up of $X'_i$ along $\tilD_i$, in particular, each $X'_i\to X$ is a modification with the trivial reduction and $\tilD_i\into X'_i\into X_i$ are monomial subschemes, and $X_{i+1}=\LogBl_{\tilD_i}(X_i)$.
\end{cor}

\subsection{Log smooth resolution}
To get our last main result it remains to assemble together the various pieces we have developed.

\begin{theor}\label{mainth}
Assume that $k$ is a field of characteristic zero, $B=\Spec(k[\pi]/(\pi^n))$ and $X\to B$ is a generically smooth morphism of finite type with a closed nowhere dense subscheme $Z\into X$. Let $M_B$ be the log structure on $B$ generated by $\bbN\log(\pi)\to\calO_B$, where $\log(\pi)$ is mapped to $\pi$. Then there exists a log smooth morphism $(Y,M_Y)\to(B,M_B)$, an irreducible component $X'$ of $Y$ and modification $X'\to X$, such that $Z'=Z\times_XX'$ is a divisor and both $X'$ and $Z'$ are monomial subschemes of $Y$. In addition, the choice of $Y$ and $X'$ can be made depending \'etale-functorially only on $X\to B$, $Z\into X$ and a generic retract $r_\eta\:\eta_X\to\eta_\tilX$, and if $X\to B$ is proper one can also choose $Y\to B$ to be proper.
\end{theor}
\begin{proof}
Acting as in the proof of Theorem~\ref{mainth1} we can find a modification $X_1\to X$ such that $X_1$ is a ptm with an ordered boundary $E_1$ and $Z_1=Z\times_XX_1$ is $E_1$-monomial. Next, the generic smoothness of $f$ implies that generically $\pi$ generates the nilradical $\calN\subset\calO_{X_1}$ on an open subscheme $U$. Principalizing $X_1\setminus U$ on $(X_1,E_1)$ by Theorem~\ref{princth} we obtain a modification $X_2\to X_1$, such that $\calN|_U=(\pi)|_U$ for $U=X_2\setminus E_2$. Next, we apply Theorem~\ref{retractth} to construct a modification $g\:X'\to X_2$ with a ptm $X'$ and an ordered boundary $E'$ containing the preimage of $E_2$ such that $X'$ possesses a retract $X'\to\tilX'$.

Set $U'=X'\setminus E'$ and $Y'=\tilX'\times_{\tilB}B=\Spec_{\tilX'}\calO_\tilX'[\pi]/(\pi^n)$. Then the natural morphism $f\:X'\to Y'$ is a modification with a trivial reduction. In addition, $f$ is an isomorphism over the complement to $\tilE'$, hence by Theorem~\ref{factorth} replacing $X'$ once again by a successive blowing up along components of $E'$, we can achieve that $f$ factors into a composition $X'=Y'_n\dashto Y'_0=Y'$ of blowings up of components of $\tilE'_i$'s, say $Y'_{i+1}=\Bl_{V_i}(Y'_i)$.

Finally, we provide $Y'$ with the log structure combined from the log structures induced from $B$ and the log structure on $\tilY'=\tilX'$ given by $\tilE'$. More formally, consider the open immersion $i\:\tilU'\into\tilY'$, then $M_{\tilY'}=i_*\calO_{\tilU'}^\times\cap\calO_{\tilY'}$ provides $\tilX'$ with the structure of a log smooth log scheme over $k$ (with the trivial log structure $k^\times\into k$) and we set $(Y',M_{Y'})=(\tilY',M_{\tilY'})\times_{\Spec(k)}(B,M_B)$. In particular, the projection $(Y,M_{Y'})\to(B,M_B)$ is log smooth. By Corollary~\ref{logblowcor} there exists a sequence of log blowings up $Y=Y_n\dashto Y_0=Y'$ such that $Y'_i\into Y_i$ is the irreducible component which is a modification of $Y'$ and the centers are log subschemes $V_i\into\tilY'=\tilY'_i\into Y'_i\into Y_i$. Thus, $X'$ is an irreducible component of $Y$ and since the log blowings up are log \'etale the composition $Y\dashto Y'\to B$ is log smooth.

The \'etale-functoriality of the construction follows from functoriality of all ingredients we used.
\end{proof}

\begin{rem}
Note for the sake of completeness that the first two steps in the above proof use the classical resolution and work only in characteristic 0, but the construction of $X'\to X_2$ (and Theorems~\ref{retractth} and \ref{factorth} used there) are purely combinatorial and work in any characteristic.
\end{rem}

\bibliographystyle{amsalpha}
\bibliography{nilresolution}

\end{document}